\documentclass[12pt]{article}

\usepackage{%fullpage,
amsthm,amsmath,amssymb,amsbsy,amsfonts}
\usepackage{amsthm,amsmath,amssymb,amsbsy,amsfonts}
\usepackage[dvipsnames,usenames]{color}
\usepackage{fullpage}
\font\SmallCap=cmcsc10
\usepackage{graphicx}
\usepackage{subfigure}
\numberwithin{equation}{section}

\title{\bf{Criteria for Invariance of Convex Bodies for Linear Parabolic Systems}}
\author{Gershon Kresin and  Vladimir Maz'ya\\ \\
{\it Dedicated to David Shoikhet on his 60th birthday}}

{ \date\ }

\numberwithin{equation}{section}
\newtheorem{lemma}{Lemma}[section]
\newtheorem{theorem}{Theorem}[section]
\newtheorem{proposition}[theorem]{Proposition}
\newtheorem{corollary}{Corollary}[section]
\newcommand{\bs}{\boldsymbol}
\begin{document}
\maketitle
%%%%%%%%%%%%%%%%%%%%%%%%%%%%%%%%%%%%%%%%%%%%%%%%%%%%%%%%%%%%%%%%%%%%%%%
{\bf Abstract.}
We consider systems of linear partial differential equations, 
which contain only second and first derivatives in the $x$ variables 
and which are uniformly parabolic in the sense of Petrovski\v{\i} in 
the layer ${\mathbb R}^n\times [0,T]$. For such systems we obtain 
necessary and, separately, sufficient conditions for invariance of a 
convex body. These necessary and sufficient conditions coincide if 
the coefficients of the system do not depend on $t$. The above 
mentioned criterion is formulated as an algebraic condition describing
a relation between the geometry of the invariant convex body and 
coefficients of the system. The criterion is concretized 
for certain classes of invariant convex sets:
polyhedral angles, cylindrical and conical bodies. 

%%%%%%%%%%%%%%%%%%%%%%%%%%%%%%%%%%%%%%%%%%%%%%%%%%%%%%%%%%%%%%%%%%%%%%
\section{Main results and background}
%%%%%%%%%%%%%%%%%%%%%%%%%%%%%%%%%%%%%%%%%%%%%%%%%%%%
We consider the Cauchy problem for parabolic systems of the form
\begin{equation} \label{(0.1)}
{\partial \bs u\over\partial t}-\sum ^n_{j,k=1}{\mathcal
A}_{jk}(x,t){\partial ^2\bs u\over\partial x_j\partial x_k}-\sum ^n_{j=1}{\mathcal
A}_j(x,t){\partial\bs u\over\partial x_j}=\bs 0,
\end{equation}
where $\bs u =(u_1,\dots, u_m)$ and $(x, t)\in {\mathbb R}_T^{n+1}={\mathbb R}^n\times (0,T]$.

By ${\mathfrak S}$ we denote the closure of an arbitrary convex  
proper subdomain of ${\mathbb R}^m$.
We say that ${\mathfrak S}$ is invariant 
for system (\ref{(0.1)}) in ${\mathbb R}_T^{n+1}$ if any solution $\bs u$ of (\ref{(0.1)}), 
which is continuous and bounded in $\overline{{\mathbb R}_T^{n+1}}$,  
belongs to ${\mathfrak S}$ 
under the assumption that $\bs u(\cdot, 0)\in {\mathfrak S}$. 
Note that the classical maximum modulus principle and the componentwise 
maximum principle for parabolic and elliptic systems can be obviously interpreted
as statements on the invariance of a ball and an orthant, respectively. 

In the present paper we are interested in algebraic conditions on the coefficients 
${\mathcal A}_{jk}, {\mathcal A}_j$ ensuring the invariance of an arbitrary 
convex ${\mathfrak S}$.

The notion of invariant set for parabolic and elliptic 
systems and the first results concerning these sets appeared in the paper 
by Weinberger \cite{WEIN}.
Nowadays, there exists a large literature on invariant sets
for nonlinear parabolic and elliptic systems with principal part subjected
to various structural conditions such as scalarity, diagonality
and others (see, for example, Alikakos \cite{ALIK1,ALIK},
Amann \cite{AM}, Bates \cite{Bates}, Bebernes and Schmitt \cite{BS},
Bebernes, Chueh and Fulks \cite{BCF}, 
Chueh, Conley and Smoller \cite{CCS}, 
Conway, Hoff and Smoller \cite{CHS},
Cosner and Schaefer \cite{CS},
Kuiper \cite{Kuip}, Lemmert \cite{Lem}, 
Redheffer and Walter \cite{RW1,RW2},
 Schaefer \cite{Shaef}, Smoller \cite{Smoller},
Weinberger \cite{WEIN1} and references there).

We note that maximum principles for weakly 
coupled parabolic systems are discussed in the books by Protter and Weinberger \cite{PW},
and Walter \cite{WAL} which also contain rich bibliographies on this subject.
The criteria on validity of the componentwise maximum principle 
for linear parabolic system of the general form
in ${\mathbb R}_T^{n+1}$ were obtained in the paper by Otsuka \cite{Ots}. 
In our papers \cite{KM1}-\cite{KM3} and \cite{MK} (see also monograph \cite{KM} 
and references therein) the criteria for validity of other type 
of maximum principles for parabolic systems were established, which are interpreted as
conditions for the invariance of compact convex bodies.

Henceforth we assume:
 
(i) real $(m\times m)$-matrix-valued functions 
${\mathcal A}_{jk}={\mathcal A}_{kj}$ and 
${\mathcal A}_j$ are defined in $\overline {{\mathbb R}_T^{n+1}}$ and have continuous and
bounded derivatives in $x$ up to the second and first order, respectively, which satisfy 
the uniform H\"older condition on $\overline {{\mathbb R}_T^{n+1}}$
with exponent $\alpha \in(0, 1]$ with respect
to the parabolic distance $\big (|x-x'|^2+|t-t'|\big )^{1/2}$;

(ii) system (\ref{(0.1)}) is uniformly parabolic in the sense of Petrovski\v{\i} 
in $\overline {{\mathbb R}_T^{n+1}}$, i.e., for any point $(x,t)\in\overline 
{{\mathbb R}_T^{n+1}}$, the real parts of the $\lambda $-roots of the equation 
$
\hbox {det}\left (\sum ^n_{j,k=1}{\mathcal A}_{jk}(x,t)\sigma _j\sigma
_k+\lambda I\right )=0
$
satisfy the inequality ${\rm Re}\; \lambda (x,t,\bs\sigma )\leq -\delta
|\bs\sigma |^2$, where $\delta $=const $>0$, for any $\bs\sigma =(\sigma
_1,\dots ,\sigma _n)\in {\mathbb R}^n,\;I$ is the identity matrix of order $m$,
and $|\cdot |$ is the Euclidean length of a vector.

The main result of the paper is the following assertion.

\smallskip
{\SmallCap Theorem.} 
(i) {\it Let the unit outward normal $\bs\nu(a)$ to $\partial{\mathfrak S}$ 
at any point $a\in \partial{\mathfrak S}$ for which it exists,
is an eigenvector of all matrices ${\mathcal A}^*_{jk}(x, t)$, ${\mathcal A}^*_j(x, t)$, 
$1\leq j,k\leq n,\;(x, t)\in {\mathbb R}^{n+1}_T$. Then ${\mathfrak S}$ is 
invariant for system $(\ref{(0.1)})$ 
in ${\mathbb R}^{n+1}_T$. Here and henceforth $^*$ means passage to the 
transposed matrix.}

(ii) {\it  Let  ${\mathfrak S}$ be invariant for system $(\ref{(0.1)})$ 
in ${\mathbb R}_T^{n+1}$. Then the unit outward normal $\bs\nu(a)$ to 
$\partial{\mathfrak S}$ at any point $a\in \partial{\mathfrak S}$ for
which it exists, is an eigenvector of all matrices ${\mathcal A}^*_{jk}(x, 0)$, 
${\mathcal A}^*_j(x, 0)$, $1\leq j,k\leq n,\;x\in {\mathbb R}^n$.}

\smallskip
We note that this result was obtained in our paper \cite{KM3} for the case
of a compact ${\mathfrak S}$ and ${\mathcal A}_j=0, 1\leq j\leq n$.
 
If the coefficients of the system do not depend on $t$, 
the theorem just formulated contains the following exhaustive criterion
of the invariance of ${\mathfrak S}$. 

\smallskip
{\SmallCap Corollary.}
{\it A convex body ${\mathfrak S}$ is invariant for parabolic system  
\begin{equation} \label{(0.1X)}
{\partial \bs u\over\partial t}-\sum ^n_{j,k=1}{\mathcal
A}_{jk}(x){\partial ^2\bs u\over\partial x_j\partial x_k}-\sum ^n_{j=1}{\mathcal
A}_j(x){\partial\bs u\over\partial x_j}=\bs 0
\end{equation}
in ${\mathbb R}_T^{n+1}$ if and only if the unit outward normal $\bs\nu(a)$ to 
$\partial{\mathfrak S}$ at any point $a\in \partial{\mathfrak S}$ for which it exists,
is an eigenvector of all matrices ${\mathcal A}^*_{jk}(x)$, ${\mathcal A}^*_j(x)$, 
$1\leq j,k\leq n,\;x\in {\mathbb R}^n$.}
   
\smallskip
We note that the conditions of smoothness of the coefficients 
of system (\ref{(0.1)}) in Theorem can be relaxed but we leave
this extension outside the scope of the present paper.

%%%%%%%%%%%%%%%%%%%%%%%%%%%%%%%%%%%%%%%%%%%%%%%%%%%%%%%%%%%%%%%%%%%
\section[Necessary conditions for invariance of a convex body]
{Necessary conditions for invariance of a convex body} 
%%%%%%%%%%%%%%%%%%%%%%%%%%%%%%%%%%%%%%%%%%%%%%%%%%%%%%%%%%%%%%%%%%%

By $[{\rm C}_{\rm b}(\overline{{\mathbb R}^{n+1}_T})]^m$
we denote the space of continuous and bounded $m$-component
vector-valued functions defined on $\overline{{\mathbb R}^{n+1}_T}$. 
By $[{\rm C}^{(2,1)}({\mathbb R}^{n+1}_T)]^m$ 
we mean the space of $m$-component vector-valued functions on ${\mathbb R}^{n+1}_T$ whose
derivatives with respect to $x$ up to the second order and first
derivative with respect to $t$ are continuous. 

\smallskip
Let $\bs\nu$ be a fixed $m$-dimensional unit vector, let 
$\bs a$ be a fixed $m$-dimensional vector, and let 
${\mathbb R}^m_{\bs \nu}(\bs a)=\{\bs u\in {\mathbb R}^m: (\bs u-\bs a, \bs \nu)\leq 0 \}$.

For the convex body ${\mathfrak S}$ by $\partial^*{\mathfrak S}$ we mean the 
set of points $a \in \partial{\mathfrak S}$ 
for which there exists the unit outward normal $\bs \nu(a)$ to $\partial{\mathfrak S}$.
We denote ${\mathfrak N}_{\mathfrak S}=\{\bs \nu(a): a \in \partial^*{\mathfrak S} \}$.

The next assertion contains a necessary condition for the invariance of a convex body
for parabolic system (\ref{(0.1)}) in ${\mathbb R}_T^{n+1}$. 

\setcounter{theorem}{0}
\begin{proposition} \label{P_1}
{\it Let a convex body ${\mathfrak S}$ 
be invariant for the system $(\ref{(0.1)})$
in ${\mathbb R}_T^{n+1}$. 
Then there exists a function $g: {\mathbb R}^{n+1}_T\times {\mathbb R}^n\times 
{\mathfrak N}_{\mathfrak S} \to {\mathbb R}$
such that
\begin{equation} \label{(3.1AB)}
G^*(t, 0 ,x, \eta )\bs \nu=g(t, x; \eta ;\bs \nu)\bs \nu\;,
\end{equation}
where $G(t, \tau ,x, \eta )$ is 
the fundamental matrix of solutions for system $(\ref{(0.1)})$.}
\end{proposition}

\begin{proof} Suppose  that ${\mathfrak S}$ is invariant 
for system (\ref{(0.1)}) in ${\mathbb R}_T^{n+1}$. 
According to Eidel'man \cite{EID1} (Theorem 1.3),  there exists a unique 
vector-valued function in 
$[{\rm C}^{(2,1)}({\mathbb R}^{n+1}_T)]^m\cap [{\rm C}_{\rm b}(\overline{{\mathbb R}^{n+1}_T})]^m$,
which satisfies the Cauchy problem
\begin{eqnarray} \label{(3.1A)}
& &{\partial \bs u\over\partial t}-\sum ^n_{j,k=1}{\mathcal
A}_{jk}(x, t){\partial ^2 \bs u \over\partial  x_j\partial x_k}-\sum ^n_{j=1}{\mathcal
A}_j(x, t){\partial \bs u \over\partial  x_j}
=\bs 0\;\;\;{\rm in}\;{\mathbb R}^{n+1}_T,\nonumber\\
& &\\
& &\bs u\big |_{t=0}=\bs\psi,\nonumber
\end{eqnarray}
where $\bs\psi$ is a bounded and continuous vector-valued function on 
${\mathbb R}^n$. This solution can be represented in the form
$$
\bs u(x,t)=\int _{{\mathbb R}^n}G(t,0, x,\eta )\bs \psi (\eta )d\eta .
$$

\smallskip
We fix a point $a\in \partial^*{\mathfrak S}$ and denote
$\bs\nu(a)$ by $\bs\nu$. Since 
\begin{equation} \label{(3.3AJ)}
\int _{{\mathbb R}^n}G(t,0, x, \eta )d\eta =I,
\end{equation}
the vector-valued function 
\begin{equation} \label{(3.3A)}
\bs u_a(x, t)=\bs u(x, t)-\bs a=
\int _{{\mathbb R}^n}G(t,0,x, \eta )\big ( \bs \psi (\eta )- \bs a \big ) d\eta 
\end{equation}
satisfies the Cauchy problem
\begin{eqnarray} \label{(3.4A)}
& &{\partial \bs u_a\over\partial t}-\!\sum ^n_{j,k=1}{\mathcal
A}_{jk}(x, t){\partial ^2 \bs u_a \over\partial  x_j\partial x_k}-\!\sum ^n_{j=1}{\mathcal
A}_j(x, t){\partial \bs  u_a \over\partial  x_j}
=\bs 0\;\;{\rm in}\;{\mathbb R}^{n+1}_T,\nonumber\\
& &\\
& &\bs u_a\big |_{t=0}=\bs\psi \!-\!\bs a.\nonumber
\end{eqnarray}

We fix a point $(x, t) \in {\mathbb R}^{n+1}_T$ and represent  
$G^*(t,0,x, \eta ) \bs \nu$ as
\begin{equation} \label{(3.7A)}
G^*(t,0,x,\eta ) \bs \nu=g(t, x; \eta; \bs\nu )\bs \nu+
\bs f(t, x; \eta; \bs\nu)\;,
\end{equation}
where 
\begin{equation} \label{(3.K)}
g(t, x; \eta; \bs\nu )=\big ( G^*(t,0,x, \eta )\bs \nu, \bs \nu \big )
\end{equation}
and
\begin{equation} \label{(3.S)}
\bs f(t, x; \eta; \bs\nu)=G^*(t,0,x,\eta ) \bs \nu 
-\big ( G^*(t,0,x, \eta )\bs \nu, \bs \nu \big )\bs\nu\;.
\end{equation}
Let us fix a point $(x, t), t>0$. 
By the boundedness and continuity in $\eta$ of $G(t,0,x,\eta )$
(see, e.g., Eidel'man \cite{EID1}, pp. 72, 93),  
$\bs f(t, x; \eta; \bs\nu)$ is also bounded and continuous in $\eta$.

Suppose there exists a set ${\mathcal M}\subset {\mathbb R}^n$, ${\rm meas}_n{\mathcal M}>0$,
such that for all $\eta \in {\mathcal M}$, the inequality
\begin{equation} \label{(3.SS)}
\bs f(t, x; \eta; \bs\nu)\neq \bs 0
\end{equation}
holds, and for all $\eta \in {\mathbb R}^n\backslash {\mathcal M}$ the equality
$\bs f(t, x; \eta; \bs\nu)= \bs 0$ is valid. 

Further, we set 
\begin{equation} \label{(3.8A)}
\bs \psi (\eta )- \bs a=\alpha\bs f(t, x; \eta; \bs\nu)-\beta\bs\nu\;,
\end{equation}
where $\alpha >0$, $\beta \geq 0$.
It follows from (\ref{(3.S)}) and (\ref{(3.8A)}) that
\begin{equation} \label{(3.8AY)} 
(\bs \psi (\eta )- \bs a, \bs \nu )=-\beta\leq 0,\;\;\;\;\;\;
|\bs \psi (\eta )- \bs a|=\big ( \alpha^2|\bs f(t, x; \eta; \bs\nu)|^2+\beta^2 \big )^{1/2}
\end{equation}
and
\begin{equation} \label{(3.8AU)}
(\bs \psi (\eta )- \bs a, G^*(t,0,x,\eta ) \bs \nu)=\alpha|\bs f(t, x; \eta; \bs\nu)|^2 - 
\beta \big ( G^*(t,0,x, \eta )\bs \nu, \bs \nu \big ).
\end{equation}

We introduce a Cartesian coordinate system ${\mathcal O}\xi_1\dots\xi_{m-1}$
in the plane, tangent to $\partial {\mathfrak S}$ with the origin at the point ${\mathcal O}=a$.
We direct the axis ${\mathcal O}\xi_m$ along the interior normal to $\partial {\mathfrak S}$.
Let $\bs e_1, \dots, \bs e_m$ denote the coordinate orthonormal basis of this system and let 
$\xi '=(\xi_1,\dots, \xi_{m-1})$.

We use the notation 
$$
\mu =\sup \{ |\bs f(t, x; \eta; \bs\nu)|: \eta \in {\mathbb R}^n \}.
$$
Let $\partial {\mathfrak S}$ be described by the equation 
$\xi_m=F(\xi')$ in a neighbourhood of ${\mathcal O}$, where $F$ 
is convex and differentiable at ${\mathcal O}$. 

We put $\beta=\max\;\{ F(\xi'): |\xi'|=\alpha\mu \}$. By (\ref{(3.8AY)}), 
$$ 
(\bs \psi (\eta )- \bs a, \bs e_m )=\beta\geq 0,\;\;\;\;|\bs \psi (\eta )- \bs a|\leq (\alpha^2\mu^2+\beta^2)^{1/2},
$$
which implies $\bs \psi (\eta ) \in {\mathfrak S}$ for all $\eta \in {\mathbb R}^n$.

By invariance of ${\mathfrak S}$, this gives 
\begin{eqnarray} \label{(3.5A)}
\big ( \bs u_a(x, t), \bs \nu \big )&=&
\int _{{\mathbb R}^n}\left ( G(t,0, x,\eta )\big ( \bs \psi (\eta )- \bs a \big ), \bs \nu\right ) d\eta\nonumber\\
& &\\
&=&\int _{{\mathbb R}^n}\left (  \bs \psi (\eta )- \bs a , G^*(t,0, x, \eta ) 
\bs \nu\right ) d\eta \leq 0\;.\nonumber
\end{eqnarray}
Now, by (\ref{(3.5A)}) and (\ref{(3.8AU)}),  
$$
0\geq\big ( \bs u_a(x, t), \bs \nu \big )=\int _{{\mathbb R}^n} \big [ \alpha|\bs f(t, x; \eta; \bs\nu)|^2 
- \beta \big ( G^*(t,0,x, \eta )\bs \nu, \bs \nu \big )\big ]d\eta, 
$$
which along with (\ref{(3.3AJ)}) leads to
\begin{equation} \label{(3.3P)}
0\geq\big ( \bs u_a(x, t), \bs \nu \big )=\alpha\left ( 
\int _{\mathcal M} |\bs f(t, x; \eta; \bs\nu)|^2 d\eta - {\beta \over \alpha}\right ).
\end{equation}

By the differentiability of $F$ at ${\mathcal O}$, we have $\beta /\alpha\rightarrow 0$ as $\alpha \rightarrow 0$.
Consequently, one can choose $\alpha$ so small that the second factor on the right-hand side of (\ref{(3.3P)})
becomes positive, which contradicts the condition ${\rm meas}_n{\mathcal M}>0$.
Therefore, $ \bs f(t, x; \eta; \bs\nu)=\bs 0$ for 
almost all $\eta \in {\mathbb R}^n$. 
This together with (\ref{(3.S)}) and the continuity of $G(t,0,x,\eta )$ in $\eta$ 
shows that $\bs f(t, x; \eta; \bs\nu)=\bs 0$ for all $\eta \in {\mathbb R}^n$.

Since $(x, t) \in {\mathbb R}^{n+1}_T$ and $a\in \partial^*{\mathfrak S}$ are arbitrary,
we arrive at (\ref{(3.1AB)}) by (\ref{(3.7A)}).
\end{proof}

We introduce the space $[{\rm C}_{\rm b}^{k,\alpha}({\mathbb R}^n )]^m$ of
$m$-component vector-valued functions defined in ${\mathbb R}^n$ and having continuous
and bounded derivatives up to order $k$, which satisfy 
the uniform H\"older condition with  exponent $\alpha , 0<\alpha\leq 1$. 

By $[{\rm C}_{\rm b}^{k,\alpha}
(\overline{{\mathbb R}^{n+1}_T})]^m$ we denote the space of $m$-component vector-valued 
functions defined in $\overline{{\mathbb R}^{n+1}_T}$, having continuous and bounded 
$x$-derivatives up to order $k$, which satisfy 
the uniform H\"older condition with  exponent $\alpha$ with respect to 
the parabolic distance $\big (|x-x'|^2+|t-t'|\big )^{1/2}$ between the 
points $(x, t)$ and $(x', t')$ in ${\mathbb R}^{n+1}_T$.  
For the space of $(m\times m)$-matrix-valued functions, defined on 
$\overline{{\mathbb R}^{n+1}_T}$ and having similar properties, we 
use the notation $[{\rm C}_{\rm b}^{k,\alpha} (\overline{{\mathbb R}^{n+1}_T})]^{m\times m}$.

\smallskip
Let
$$
{\mathfrak A}(x, t,D_x)=\sum ^n_{j,k=1}{\mathcal
A}_{jk}(x, t){\partial ^2  \over\partial  x_j\partial x_k}+\sum ^n_{j=1}{\mathcal
A}_{j}(x, t){\partial  \over\partial  x_j}+{\mathcal A}_0(x, t).
$$
We quote the following known assertion (see Eidel'man \cite{EID1}, Theorem 5.3), 
which will be used in the sequel.

\setcounter{theorem}{0}
\begin{theorem} \label{T_1} {\it Let $(m\times m)$-matrix valued coefficients 
${\mathcal A}_{jk}, {\mathcal A}_j, {\mathcal A}_0$ 
of the operator ${\mathfrak A}(x, t,D_x)$ belong to 
$[{\rm C}_{\rm b}^{0,\alpha} (\overline{{\mathbb R}^{n+1}_T})]^{m\times m}$ and let 
$\bs u_0\in [{\rm C}_{\rm b}^{2,\alpha}({\mathbb R}^n )]^m$. Let, further, the system 
$$
{\partial \bs u \over \partial t}-{\mathfrak A}(x, t,D_x)\bs u=\bs 0,
$$
$\bs u=(u_1,\dots,u_m)$, be uniformly parabolic in the sense of Petrovski\v{\i} in the layer 
$\overline{{\mathbb R}^{n+1}_T}$ and let $G(t,\tau, x,\eta )$ be its fundamental matrix. 
 
Then the vector-valued function
$$
\bs u(x,t)=\int _{{\mathbb R}^n}G(t,0, x,\eta )\bs u_0(\eta )d\eta 
$$
belongs to $[{\rm C}_{\rm b}^{2,\alpha}(\overline{{\mathbb R}^{n+1}_T})]^m$ and it
is a unique solution of the Cauchy problem}
$$
{\partial \bs u \over \partial t}-{\mathfrak A}(x, t,D_x)\bs u=\bs 0\;\;\;{\rm in}\;
{\mathbb R}^{n+1}_T,\;\;\;\;\bs u\big |_{t=0}=\bs u_0\;.
$$
\end{theorem}

The following assertion gives a necessary condition for the invariance of ${\mathfrak S}$
which is formulated in terms of the coefficients of system (\ref{(0.1)}). It settles the 
necessity part of Theorem from Sect. 1.

\setcounter{theorem}{1}
\begin{proposition} \label{P_2}
{\it Let a convex body ${\mathfrak S}$ %\subset {\mathbb R}^m$ 
be invariant for system $(\ref{(0.1)})$ in ${\mathbb R}_T^{n+1}$. Then there exist
functions $a _{jk},\;a _j:{\mathbb R}^n\times {\mathfrak N}_{\mathfrak S}
\to {\mathbb R},\;1\leq j,k\leq n$, such that
$$
{\mathcal A}^*_{jk}(x, 0)\bs \nu =a_{jk}(x; \bs\nu)\bs \nu,\;\;\;
{\mathcal A}^*_j(x, 0)\bs \nu =a_j(x; \bs\nu)\bs \nu. %,\;\;\;1\leq j,k\leq n.
$$
}
\end{proposition}

\begin{proof} Suppose  that ${\mathfrak S}$ is invariant for system (\ref{(0.1)}) 
in ${\mathbb R}_T^{n+1}$. We fix a point $a\in \partial^*{\mathfrak S}$ and denote
$\bs\nu(a)$ by $\bs\nu$.
Let the function $\bs \psi$ in (\ref{(3.4A)}) is defined by
\begin{equation} \label{(3.3Y)}
\bs\psi(x)=\bs a+\left (\sum_{j,k=1}^n \alpha_{jk}(x_j-y_j)(x_k-y_k)+ 
\sum_{j=1}^n\beta_j(x_j-y_j) \right )\zeta_r (x-y)\bs \tau,
\end{equation}
where $\alpha_{jk}, \beta_{j}$ are constants, $y$ is a fixed point in ${\mathbb R}^n$, 
$\zeta_r \in {\rm C}^\infty_0({\mathbb R}^n)$, $0\leq\zeta_r (x)\leq 1$, 
$\zeta_r(x)=1$ for $|x|\leq r/2$ and $\zeta_r(x)=0$ for $|x|\geq r$,
$\bs\tau$ is a unit $m$-dimensional vector which is orthogonal to $\bs\nu$. 

It follows from (\ref{(3.3A)}) 
and Proposition \ref{P_1} that
\begin{eqnarray*}
\big ( \bs u_a(x, t), \bs \nu \big )&=&
\int _{{\mathbb R}^n}\left (  \bs \psi (\eta )- \bs a , G^*(t,0, x, \eta ) 
\bs \nu\right ) d\eta \\
& &\\
&= &\int _{{\mathbb R}^n}g(t, x; \eta ;\bs \nu)
\left ( \bs \psi (\eta )- \bs a , \;\bs \nu\right ) d\eta\;,
\end{eqnarray*}
which, by (\ref{(3.3Y)}), gives $\big ( \bs u_a(x, t), \bs \nu \big )=0$.
This and (\ref{(3.4A)}) imply
$$
\sum ^n_{j,k=1}\left ( {\partial ^2 \bs u_a \over\partial  x_j\partial x_k},\; 
{\mathcal A}_{jk}^*(x, t)\bs \nu\right )+
\sum ^n_{j=1}\left ({\partial \bs  u_a \over\partial  x_j},\; 
{\mathcal A}_j(x, t)^* \bs \nu\right )=\bs 0\;.
$$
By Theorem \ref{T_1},
we pass to the limit as $t\rightarrow 0$ to obtain
\begin{equation} \label{(3.3H)}
\sum ^n_{j,k=1}\left ( {\partial ^2 \bs \psi_a \over\partial  x_j\partial x_k},\; 
{\mathcal A}_{jk}^*(x, 0)\bs \nu\right )+
\sum ^n_{j=1}\left ({\partial \bs \psi_a \over\partial  x_j},\; 
{\mathcal A}_j^*(x, 0) \bs \nu\right )=\bs 0\;,
\end{equation}
where $\bs\psi_a(x)=\bs\psi(x)-\bs a$. Now, (\ref{(3.3Y)}) leads to
$$
{\partial ^2 \bs \psi_a \over\partial  x_j\partial x_k}\Big |_{x=y}=\alpha_{jk}\bs \tau\;,\;\;\;\;\;\;\;
{\partial \bs \psi_a \over\partial  x_j}\Big |_{x=y}=\beta_j\bs \tau\;.
$$
Then, by (\ref{(3.3H)}),
$$
\sum ^n_{j,k=1}\alpha_{jk}\Big (\bs \tau ,\; 
{\mathcal A}_{jk}^*(y, 0)\bs \nu\Big )+
 \sum ^n_{j=1}\beta_j\Big (\bs \tau,\; {\mathcal A}_j^*(y, 0) \bs \nu\Big )=\bs 0\;.
$$
Hence, by arbitrariness of $\alpha_{jk}, \beta_{j}$ and $\bs\tau$, 
we arrive at the equalities
$$
{\mathcal A}^*_{jk}(y, 0)\bs \nu =a_{jk}(y)\bs \nu,\;\;\;
{\mathcal A}^*_j(y, 0)\bs \nu =a_j(y)\bs \nu,\;\;\;1\leq j,k\leq n
$$
with $\bs\nu=\bs\nu(a)$, where $y\in {\mathbb R}^n$ and $a\in \partial^*{\mathfrak S}$ 
are arbitrary fixed points. The proof is complete. 
\end{proof}

%%%%%%%%%%%%%%%%%%%%%%%%%%%%%%%%%%%%%%%%%%%%%%%%%%%%%%%%%%%%%%%%%%%%%%%%%%%
\section[Sufficient condition for invariance of a convex body]  
{Sufficient condition for invariance of a convex body}
%%%%%%%%%%%%%%%%%%%%%%%%%%%%%%%%%%%%%%%%%%%%%%%%%%%%%%%%%%%%%%%%%%%%%%%%%%%

Let $\bs\nu$ be a fixed $m$-dimensional unit vector and let $a$ stand for a fixed point in ${\mathbb R}^m$. 

\begin{proposition} \label{P_4}
{\it Let the equalities
\begin{equation} \label{(4.1)}
{\mathcal A}^*_{jk}(x,t)\bs\nu =a _{jk}(x,t)\bs\nu,\;\;\;
{\mathcal A}^*_j(x,t)\bs\nu =a _j(x,t)\bs\nu ,\;\;\;1\leq j,k\leq n,
\end{equation}
hold for all $(x,t)\in {\mathbb R}^{n+1}_T$  with 
$a_{jk},\;a_j:{\mathbb R}^{n+1}_T\to {\mathbb R}.$ 
Then the half-space ${\mathbb R}^m_{\bs \nu}(\bs a)$ 
is an invariant set for system $(\ref{(0.1)})$
in ${\mathbb R}^{n+1}_T$.}
\end{proposition}

\begin{proof} Let $\bs u\in [{\rm C}_{\rm b}(\overline {{\mathbb R}^{n+1}_T})]^m\cap 
[{\rm C}^{(2,1)}({\mathbb R}^{n+1}_T)]^m$ be a solution
of the Cauchy problem (\ref{(3.1A)}).
Then the vector-valued function $\bs u_a=\bs u- \bs a$ is solution of the Cauchy problem (\ref{(3.4A)}).

Hence,
\begin{eqnarray*}
& &{\partial\over\partial t}(\bs u_a,\bs \nu)-\sum
^n_{j,k=1}\left ({\mathcal A}_{jk}(x,t){\partial ^2\bs u_a\over\partial
x_j\partial x_k},\;\bs\nu \right )-\sum
^n_{j=1}\left ({\mathcal A}_j(x,t){\partial \bs u_a\over\partial
x_j},\;\bs\nu \right )\\
& &={\partial\over\partial t}(\bs u_a,\bs\nu )-\sum
^n_{j,k=1}\left ({\partial ^2\bs u_a\over\partial x_j\partial
x_k},\;{\mathcal A}^*_{jk}(x,t)\bs\nu\right )-\sum
^n_{j=1}\left ({\partial \bs u_a\over\partial x_j\partial
x_k},\;{\mathcal A}^*_j(x,t)\bs\nu\right )=0\;.
\end{eqnarray*}
By (\ref{(4.1)}) we arrive at
\begin{eqnarray*} 
& &{\partial\over\partial t}(\bs u_a,\bs\nu)-\sum
^n_{j,k=1}\left ({\partial ^2\bs u_a\over\partial x_j\partial
x_k},\;a_{jk}(x,t)\bs\nu\right )-\sum
^n_{j=1}\left ({\partial \bs u_a\over\partial x_j},\;a_j(x,t)\bs\nu\right )\\
& &={\partial\over\partial t}(\bs u_a,\bs\nu )-\sum
^n_{j,k=1}a_{jk}(x,t){\partial ^2\over\partial x_j\partial
x_k}(\bs u_a,\bs\nu )-\sum
^n_{j=1}a_j(x,t){\partial \over\partial x_j}(\bs u_a,\bs\nu )=0\;.
\end{eqnarray*}
Thus the function $u_a=(\bs u_a,\bs\nu )$ satisfies 
$$
{\partial u_a\over\partial t}-\sum ^n_{j,k=1}a_{jk}(x,t)
{\partial ^2 u_a\over\partial x_j\partial x_k}
-\sum ^n_{j=1}a_j(x,t){\partial u_a\over\partial x_j}
=0\;\;\hbox {in}\;\;{\mathbb R}^{n+1}_T,\;u_a\Big |_{t=0}=(\bs\psi -\bs a, \bs\nu ).
$$
Therefore, by the maximum principle for solutions to the scalar parabolic
equation in ${\mathbb R}^{n+1}_T$ with the unknown function $u_a$, we conclude
$$
\inf_{y\in{\mathbb R}^n}\big (\bs u(y,0)-\bs a, \bs \nu \big )
\leq \big ( \bs u(x,t)-\bs a, \bs \nu \big ) \leq \;
\sup_{y\in{\mathbb R}^n}\big (\bs u(y,0)-\bs a, \bs \nu \big ),
$$
i.e., the half-space ${\mathbb R}^m_{\bs \nu}(\bs a)$ is invariant for system (\ref{(0.1)})
in ${\mathbb R}^{n+1}_T$.
\end{proof}

The next assertion results directly from Proposition \ref{P_4} 
and the known assertion (Rockafellar \cite{ROCK}, Theorem 18.8):
$$
{\mathfrak S}=\bigcap_{a \in \partial^*{\mathfrak S}}{\mathbb R}^m_{\bs \nu(a)}(\bs a).
$$

\begin{proposition} \label{P_5}
{\it Let ${\mathfrak S}$ 
be a convex body and let the equalities
$$
{\mathcal A}^*_{jk}(x,t)\bs\nu =a _{jk}(x,t;\bs\nu)\bs\nu,\;\;\;
{\mathcal A}^*_j(x,t)\bs\nu =a _j(x,t;\bs\nu)\bs\nu,\;\;\;1\leq j,k\leq n,
$$
hold for all $(x,t)\in {\mathbb R}^{n+1}_T$  and $\bs\nu\in{\mathfrak N}_{\mathfrak S}$ with 
$a_{jk},\;a_j:{\mathbb R}^{n+1}_T\times {\mathfrak N}_{\mathfrak S}\to {\mathbb R}$.

Then ${\mathfrak S}$ is an invariant for system $(\ref{(0.1)})$
in ${\mathbb R}^{n+1}_T$.}
\end{proposition}

Hence, the proof of sufficiency in Theorem from Sect.1 is obtained.

%%%%%%%%%%%%%%%%%%%%%%%%%%%%%%%%%%%%%%%%%%%%%%%%%%%%%%%%%%%%%%%%%%%%%%%%%%%
\section[Corollaries]{Corollaries}
%%%%%%%%%%%%%%%%%%%%%%%%%%%%%%%%%%%%%%%%%%%%%%%%%%%%%%%%%%%%%%%%%%%%%%%%%%%

\smallskip
Let us introduce a layer ${\mathbb R}^{n+1}_{\tau, T}={\mathbb R}^n\times (\tau, T]$, 
where $\tau \in [0, T)$. We say that a convex body ${\mathfrak S}$ is invariant
for system (\ref{(0.1)}) in ${\mathbb R}_{\tau, T}^{n+1}$, if any solution $\bs u$ 
of (\ref{(0.1)}), which is continuous and bounded 
in $\overline{{\mathbb R}_{\tau, T}^{n+1}}$, 
belongs to ${\mathfrak S}$ under the assumption that its initial values 
$\bs u(\cdot, \tau)$ lie in ${\mathfrak S}$.

Let $\tau \in [0, T)$. Repeating almost word for word all previous proofs replacing
$\bs u|_{t=0}$ by $\bs u|_{t=\tau}$, ${\mathbb R}^{n+1}_{0, T}$ by 
${\mathbb R}^{n+1}_{\tau, T}$, $G(t,0,x, \eta )$ by $G(t,\tau,x, \eta )$ 
and making obvious similar changes, we arrive 
at the following criterion for the invariance of ${\mathfrak S}$ for the parabolic system 
(\ref{(0.1)}) in any layer ${\mathbb R}^{n+1}_{\tau, T}$ with $\tau \in [0, T)$.

\begin{proposition} \label{T_6}
{\it A convex body ${\mathfrak S}$ is invariant for system $(\ref{(0.1)})$
in the layer ${\mathbb R}^{n+1}_{\tau, T}$ for all $\tau \in [0, T)$ simultaneously,
if and only if the unit outward normal $\bs\nu(a)$ to 
$\partial{\mathfrak S}$ at any point $a\in \partial{\mathfrak S}$ for which it exists,
is an eigenvector of all matrices ${\mathcal A}^*_{jk}(x, t),\; {\mathcal A}^*_j(x, t)$, 
$1\leq j,k\leq n,\;(x, t)\in {\mathbb R}^{n+1}_T$.} 
\end{proposition}

All criteria, formulated below, concern invariant convex bodies
for system (\ref{(0.1X)}) in ${\mathbb R}^{n+1}_T$. We note that 
similar assertions are valid also for system (\ref{(0.1)}) in any layer 
${\mathbb R}^{n+1}_{\tau, T}$ with $\tau \in [0, T)$.

\smallskip
{\bf Polyhedral angles.} We introduce a polyhedral angle 
$$
{\mathbf R}^m_+(\alpha_{m-k+1}, \dots,\alpha_m)=\{ u=(u_1,\dots, u_m): u_{m-k+1}\geq \alpha_{m-k+1}, \dots,u_m\geq \alpha_m \},
$$
where $k=1,\dots, m$. In particular, ${\mathbf R}^m_+(\alpha_m)$ is a half-space, 
${\mathbf R}^m_+(\alpha_{m-1},\alpha_m)$ is a dihedral angle, and ${\mathbf R}^m_+(\alpha_1,\dots\alpha_m)$ 
is an orthant in ${\mathbb R}^m$.

Using Corollary stated in Sect. 1, we derive 

\begin{corollary} \label{C_1} 
{\it The polyhedral angle ${\mathbf R}^m_+(\alpha_{m-k+1}, \dots,\alpha_m)$ 
is invariant for system $(\ref{(0.1X)})$ in ${\mathbb R}^{n+1}_T$ if and 
only if all nondiagonal elements of $m-k+1$-th,\dots, $m$-th rows of the 
matrix-valued functions ${\mathcal A}_{jk}$ and ${\mathcal A}_j$,
$1\leq j,k\leq n$, are equal to zero. 

In particular, a half-plane ${\mathbf R}^2_+(\alpha_2)$ is invariant 
for system $(\ref{(0.1X)})$ in ${\mathbb R}^{n+1}_T$ if and only if all
$(2\times 2)$-matrix-valued functions ${\mathcal A}_{jk}$ and ${\mathcal A}_j$,
$1\leq j,k\leq n$, are upper triangular. 
}
\end{corollary}  

{\bf Cylinders.} Let 
$$
{\mathbf R}^m_-(\beta_{m-k+1}, \dots,\beta_m)=\{ u=(u_1,\dots, u_m): 
u_{m-k+1}\leq \beta_{m-k+1}, \dots,u_m\leq \beta_m \}
$$
be a polyhedral angle and $\alpha_{m-k+1}<\beta_{m-k+1}, \dots, \alpha_m<\beta_m$.

Let us introduce a polyhedral cylinder 
$$
{\mathbf C}^m(\alpha_{m-k+1}, \dots,\alpha_m; \beta_{m-k+1}, \dots,\beta_m)=
{\mathbf R}^m_+(\alpha_{m-k+1}, \dots,\alpha_m)\cap{\mathbf R}^m_-(\beta_{m-k+1}, 
\dots,\beta_m),
$$
$k<m$.

In particular, ${\mathbf C}^m(\alpha_m; \beta_m)$ is a layer and 
${\mathbf C}^m(\alpha_{m-1},\alpha_m; \beta_{m-1},\beta_m)$ is a rectangular cylinder.

The following criterion stems from Corollary stated in Sect. 1. 

\begin{corollary} \label{C_2}
{\it The polyhedral cylinder ${\mathbf C}^m(\alpha_{m-k+1}, 
\dots,\alpha_m; \beta_{m-k+1}, \dots,\beta_m)$ 
is invariant for system $(\ref{(0.1X)})$ in ${\mathbb R}^{n+1}_T$ 
if and only if all nondiagonal elements 
of $m-k+1$-th, $m-k+2$-th,\dots, $m$-th rows of matrix-valued functions 
${\mathcal A}_{jk}$ and ${\mathcal A}_j$,
$1\leq j,k\leq n$, are equal to zero. 

In particular, a strip ${\mathbf C}^2(\alpha_2; \beta_2)$ is invariant 
for system $(\ref{(0.1X)})$ in ${\mathbb R}^{n+1}_T$ if and only if all
$(2\times 2)$-matrix-valued functions ${\mathcal A}_{jk}$ and ${\mathcal A}_j$,
$1\leq j,k\leq n$, are upper triangular. 
}
\end{corollary}

Let us introduce the body
$$
{\mathbf S}^m_k(R)=\{ u=(u_1,\dots, u_m): u_{m-k+1}^2+\dots +u_m^2\leq R^2 \},
$$
which is a spherical cylinder for $k< m$. 

Using Corollary stated in Sect. 1, we arrive at the following criterion.

\begin{corollary} \label{C_3} 
{\it The body ${\mathbf S}^m_k(R)$ 
is invariant for system $(\ref{(0.1X)})$ in ${\mathbb R}^{n+1}_T$ if and only if:
 
{\rm (i)} all nondiagonal elements of $m-k+1$-th, $m-k+2$-th,\dots, $m$-th rows of
matrix-valued functions ${\mathcal A}_{jk}$ and ${\mathcal A}_j$,
$1\leq j,k\leq n$, are equal to zero; 

{\rm (ii)} all $m-k+1$-th, $m-k+2$-th,\dots, $m$-th diagonal elements of matrix 
${\mathcal A}_{jk}(x)$ $\big ({\mathcal A}_j(x) \big )$ are equal for any 
fixed point $x\in {\mathbb R}^n$ and indices $j,k=1,\dots,n$.}
\end{corollary}

\smallskip
{\bf Cones.} By ${\mathbf K}^m_p$ we denote a convex polyhedral 
cone in ${\mathbb R}^m$ with $p$ facets. Let, further, $\{\bs\nu_1,\dots, \bs\nu_p \}$ be 
the set of unit outward normals to the facets of this cone. 
By $[\bs v_1,\dots, \bs v_m]$ we mean the
$(m\times m)$-matrix whose columns are $m$-component vectors $\bs v_1,\dots, \bs v_m$.

\smallskip
We give an auxiliary assertion of geometric character. 

\setcounter{theorem}{0}
\begin{lemma} \label{L_1} Let ${\mathbf K}^m_p$ be a convex polyhedral cone in
${\mathbb R}^m$ with $p$ facets, $p\geq m$. 
Then any system $\bs\nu_1,\dots,\bs\nu_m $ of unit outward normals to 
$m$ different facets of ${\mathbf K}^m_p$ is linear independent.
\end{lemma}
\begin{proof} 
By $F_i$ we denote the facet of ${\mathbf K}^m_p$ for which the 
vector $\bs\nu_i$ is normal, $1\leq i\leq m$. Let $T_i$ be the supporting 
plane of this facet. We place the origin
of the coordinate system with the orthonormal basis $\bs e_1,\dots ,\bs e_m$ at
an interior point ${\mathcal O}$ of ${\mathbf K}^m_p$ and use the notation 
$x={\mathcal O}q$, where $q$ is the vertex of the cone. Further, 
let $d_i=\mbox{dist}\;({\mathcal O}, F_i),\;i=1,\dots ,m$. Since
$$
q=\bigcap ^m_{i=1}T_i\;,
$$
it follows that $x=(x_1,\dots ,x_m)$ is the only solution of the system
$$
(\bs\nu_i ,x)=d_i,\;\;i=1,2,\dots ,m,
$$
or, which is the same,
\[
\sum ^m_{j=1}(\bs\nu_i ,\bs e_j)x_j=d_i,\;\;i=1,2,\dots ,m.
\]
The matrix of this system is $[\bs\nu_1,\dots ,\bs\nu_m]^*$. Consequently, 
$$
\det[\bs\nu_1,\dots ,\bs\nu_m]^*\not= 0. 
$$
This implies the linear independence of the system $\bs\nu_1,\dots ,\bs\nu_m$.
\end{proof}

\begin{corollary} \label{C_4} 
{\it The convex polyhedral cone ${\mathbf K}^m_m$
is invariant for system $(\ref{(0.1X)})$ in ${\mathbb R}^{n+1}_T$ if and only if 
\begin{equation} \label{K_1}
{\mathcal A}_{jk}(x)=\big ([\bs\nu_1,\dots, \bs\nu_m]^* \big)^{-1}\;
{\mathcal D}_{jk}(x)\;[\bs\nu_1,\dots, \bs\nu_m]^*
\end{equation}
and
\begin{equation} \label{K_2}
{\mathcal A}_j(x)=\big ([\bs\nu_1,\dots, \bs\nu_m]^* \big)^{-1}\;{\mathcal D}_j(x)\;
[\bs\nu_1,\dots, \bs\nu_m]^*
\end{equation}
for all $x\in {\mathbb R}^n$, $1\leq j,k\leq n$, where ${\mathcal D}_{jk}$ 
and ${\mathcal D}_j$ are diagonal $(m\times m)$-matrix-valued functions.

The convex polyhedral cone ${\mathbf K}^m_p$ with $p>m$ and convex cone with a smooth guide
are invariant for system $(\ref{(0.1X)})$ in ${\mathbb R}^{n+1}_T$ if and only if all
matrix-valued functions ${\mathcal A}_{jk}$ and ${\mathcal A}_j$, $1\leq j,k\leq n$, 
are scalar.}
\end{corollary}

\begin{proof} 
We fix a point $x \in {\mathbb R}^n$. By ${\mathcal A}$ we denote any 
of the $(m\times m)$-matrices ${\mathcal A}_{jk}(x)$ and
${\mathcal A}_j(x)$, $1\leq j,k\leq n$. 

By Corollary stated in Sect. 1, a necessary and sufficient 
condition for invariance of ${\mathfrak S}$ is equation
\begin{equation} \label{K_3}
{\mathcal A}^*\bs \nu=\mu\bs\nu\;\;\mbox{for any}\;\;\bs\nu\in {\mathfrak N}_{\mathfrak S},
\end{equation}
where $\mu=\mu(\bs\nu)$ is a real number.

\medskip
(i) If ${\mathfrak S}={\mathbf K}^m_m$, we write (\ref{K_3}) as
\begin{equation} \label{K_4}
{\mathcal A}^*\bs \nu_1=\mu_1\bs \nu_1, 
\dots,{\mathcal A}^*\bs \nu_m=\mu_m\bs \nu_m\;,
\end{equation}
where $\{\bs\nu_1,\dots,\bs \nu_m \}$ is the set of unit outward normals 
to the facets of the ${\mathbf K}^m_m$. These normals are linear independent 
by Lemma \ref{L_1}. Let ${\mathcal D}=\mbox{diag}\;\{ \mu_1,\dots,\mu_m \}$. 
Equations (\ref{K_4}) can be written as
$$
{\mathcal A}^*[\bs \nu_1,\dots,\bs \nu_m]=[\bs \nu_1,\dots,\bs \nu_m]\;{\mathcal D},
$$
which leads to the representation
\begin{equation} \label{K_5}
{\mathcal A}=\big ([\bs \nu_1,\dots,\bs \nu_m]^*\big )^{-1}\;{\mathcal D}\;
[\bs \nu_1,\dots,\bs \nu_m]^*\;.
\end{equation}
Now, (\ref{K_5}) is equivalent to (\ref{K_1}) and (\ref{K_2}). 

\medskip
(ii) Let us consider the cone ${\mathbf K}^m_p$ with $p>m$. 
By $\{\bs\nu_1,\dots,\bs \nu_m \}$ we denote a system of unit outward normals to
$m$ facets of ${\mathbf K}^m_p$. Let also $\bs \nu$ be a normal to a certain $m\!+\!1$-th facet.
By Lemma \ref{L_1}, arbitrary $m$ vectors in the collection 
$\{\bs\nu_1,\dots,\bs \nu_m, \bs\nu \}$ are linear independent.
Hence there are no zero coefficients $\alpha_i$ in the representation
$\bs \nu=\alpha_1\bs\nu_1+\dots+\alpha_m\bs\nu_m$.

Let (\ref{K_3}) hold. Then
\begin{equation} \label{K_6}
{\mathcal A}^*\bs \nu=\lambda\bs \nu,\;{\mathcal A}^*\bs \nu_1=\mu_1\bs \nu_1, 
\dots,{\mathcal A}^*\bs \nu_m=\mu_m\bs \nu_m\;.
\end{equation}
Therefore,
$$
\lambda\sum_{i=1}^m\alpha_i \bs\nu_i=\lambda\bs\nu={\mathcal A}^*\bs \nu={\mathcal A}^*
\sum_{i=1}^m\alpha_i \bs\nu_i=\sum_{i=1}^m\alpha_i \mu_i\bs\nu_i.
$$
Thus,
$$
\sum_{i=1}^m (\lambda-\mu_i)\alpha_i \bs\nu_i=\bs 0. 
$$
Hence, $\mu_i=\lambda$ for $i=1,\dots,m$ and consequently ${\mathcal A}$ is a scalar matrix.

Conversely, if ${\mathcal A}=\lambda\;\mbox{diag}\;\{1 ,\dots, 1 \}$, then (\ref{K_3}) 
with $\mu=\lambda$ holds  for ${\mathfrak S}={\mathbf K}^m_p$ with $p>m$.

The proof is complete for $p>m$.

\medskip
(iii) Let (\ref{K_3}) hold for the cone ${\mathbf K}$ with a smooth guide.
This cone ${\mathbf K}$ can be inscribed into a polyhedral cone 
${\mathbf K}^m_{m+1}$. Let $\{\bs\nu_1,\dots,\bs \nu_m, \bs\nu \}$ be a system 
of unit outward normals to the facets of ${\mathbf K}^m_{m+1}$. 
This system is a subset of the collection of normals to the boundary of ${\mathbf K}$. 
By Lemma \ref{L_1}, arbitrary $m$ vectors in the set 
$\{\bs\nu_1,\dots,\bs \nu_m, \bs\nu \}$ are linear independent.
Repeating word by word the argument used in (ii) we arrive at the
scalarity of ${\mathcal A}$.

Conversely, (\ref{K_3}) is an obvious consequence of the scalarity of 
${\mathcal A}$ for ${\mathfrak S}={\mathbf K}$.

The proof is complete.
\end{proof}
%%%%%%%%%%%%%%%%%%%%%%%%%%%%%%%%%%%%%%%%%%%%%%%%%%%%%%%%%%%%%%%%%%

\bibliographystyle{amsplain}

%%%%%%%%%%%%%%%%%%%%%%%%%%%%%%%%%%%%%%%%%%%%%%%%%%%%%%%%%%%%%%%%
\bigskip
\bigskip
{
\sc Gershon Kresin}
\\
{\sc e-mail:\;}{\tt kresin@ariel.ac.il%, kresin@netvision.net.il
}
\\
{\sc address:\;}{\it Department of Computer Science and Mathematics,
            Ariel University,\\
            44837 Ariel,
            Israel}
\\
\\
{\sc Vladimir Maz'ya}
\\
{\sc e-mail:\;}{\tt vladimir.mazya@liu.se}
\\
{\sc address:\;}{\it {Department of Mathematical Sciences,
            University of Liverpool,
            M$\&$O Building,\\ Liverpool, L69 3BX,
            UK}
%\\            
%{\sc e-mail:\;}{\tt vladimir.mazya@liu.se}
\\
{\sc address:\;}{\it Department of Mathematics,
            Link\"oping University,
            SE-58183 Link\"oping,
            Sweden}        
%%%%%%%%%%%%%%%%%%%%%%%%%%%%%%%%%%%%%%%%%%%%%%%%%%%%%%%%%%%%%%%%%%%%%
\end{document}